\documentclass[10pt]{article}

\usepackage{latexsym,amssymb}

\begin{document}

\title{\bf Full Orientability of the Square \\ of a Cycle}

\author{
Fengwei Xu, \  Weifan Wang
\thanks{\footnotesize  Corresponding author. Email: wwf@zjnu.cn;
Research supported partially by  NSFC (No. 10771197) and ZJNSF
(No. Z6090150)}\\
\normalsize Department of Mathematics\\
\normalsize Zhejiang Normal University, Jinhua 321004, China 
\and
Ko-Wei Lih\\
\normalsize Institute of Mathematics\\
\normalsize Academia Sinica,  Nankang, Taipei 11529, Taiwan}

\date{}

\maketitle

\newtheorem{define}{Definition}
\newtheorem{proposition}[define]{Proposition}
\newtheorem{theorem}[define]{Theorem}
\newtheorem{lemma}[define]{Lemma}
\newtheorem{remark}[define]{Remark}
\newtheorem{corollary}[define]{Corollary}
\newtheorem{problem}[define]{Problem}
\newtheorem{conjecture}[define]{Conjecture}
%
%
\newenvironment{proof}{
\par
\noindent {\bf Proof.}\rm}%
{\mbox{}\hfill\rule{0.5em}{0.809em}\par}

\baselineskip=16pt
\parindent=0.5cm


\begin{abstract}

\baselineskip=14pt

 \noindent Let $D$ be an acyclic orientation of
a simple graph $G$. An arc of $D$ is called {\em dependent} if its
reversal creates a directed cycle. Let $d(D)$ denote the number of
dependent arcs in $D$. Define $d_{\min}(G)$ ($d_{\max}(G)$) to be
the minimum (maximum) number of $d(D)$ over all acyclic
orientations $D$ of $G$. We call $G$ {\em fully orientable} if $G$
has an acyclic orientation with exactly $k$ dependent arcs for
every $k$ satisfying $d_{\min}(G)\leqslant k\leqslant
d_{\max}(G)$. In this paper, we prove that the square of a cycle
$C_n$ is fully orientable except $n=6$.

\medskip


\medskip

\end{abstract}

\medskip

%
\section{Introduction}
%
\baselineskip=14pt

Only simple graphs are considered in this paper unless otherwise
stated. For a  graph $G$, we denote its vertex set and  edge set
by $V(G)$ and $E(G)$, respectively. An {\em orientation} $D$ of
$G$ assigns a direction to each edge of $G$. $D$ is called {\em
acyclic}  if there does not exist any directed cycle. Suppose that
$D$ is an acyclic orientation of $G$. An arc of $D$ is called {\em
dependent} if its reversal creates a directed cycle. Let $d(D)$
denote the number of dependent arcs of $D$. We use $d_{\min}(G)$
and $d_{\max}(G)$ to denote the minimum and maximum number of $d(D)$
over all acyclic orientations $D$ of $G$, respectively. It is known
\cite{fish}  that $d_{\max}(G)=|E(G)|-|V(G)|+ c$ for a graph $G$
having $c$ components.

An interpolation question asks whether $G$ has an acyclic
orientation with exactly $k$ dependent arcs for each $k$
satisfying $d_{\min}(G)\leqslant k\leqslant d_{\max}(G)$.
The graph $G$ is called {\em fully orientable} if its
interpolation question has an affirmative answer.  West
\cite{west} showed that complete bipartite graphs are
fully orientable.

A $k$-{\em coloring} of a graph $G$ is a mapping $f$ from
$V(G)$ to the  set $\{1,2,\ldots , $ $k\}$ such that $f(x)\ne
f(y)$ for each edge $xy\in E(G)$. The {\em chromatic number}
$\chi(G)$ is the smallest integer $k$ such that $G$ has a
$k$-coloring. The {\em girth} $g(G)$ is the minimum length
of a cycle in a graph $G$ if there is any, and $\infty$ if
$G$ possesses no cycles.

Fisher et al. \cite{fish} showed that $G$ is fully orientable if
$\chi(G) < g(G)$, and $d_{\min}(G)=0$ in this case.  Since it is
well-known \cite{gr} that every planar graph $G$ with
$g(G)\geqslant 4$ is 3-colorable, planar graphs of girth at least
4 are fully orientable.

The full orientability for several graph classes has been
investigated recently. Lih, Lin, and Tong \cite{lih1} showed that
outerplanar graphs are fully orientable. To generalize this
result, Lai, Chang, and Lih \cite{lai1}  proved  that 2-degenerate
graphs are fully orientable. Here a graph $G$ is called  2-{\em
degenerate} if every subgraph $H$ of $G$ contains a vertex of
degree at most 2 in $H$. Lai and Lih \cite{lai2} gave further
examples of fully orientable graphs, such as subdivisions of Halin
graphs and graphs of maximum degree at most three. Let $K_{r(n)}$
denote the complete $r$-partite graph each of whose partite sets
has $n$ vertices. Chang, Lin, and Tong \cite{chan} proved that
$K_{r(n)}$ is not fully orientable if $r\geqslant 3$ and
$n\geqslant 2$. These are the only known graphs that are not fully
orientable.

Suppose that $G$ is a connected graph. For $m\geqslant 2$, the
$m$th {\em power} of $G$, denoted $G^m$, is the graph defined by
$V(G^m)=V(G)$ and two distinct vertices $u$ and $v$ are adjacent
in $G^m$ if and only if their distance in $G$ is at most $m$. In
particular, $G^2$ is called the {\em square} of $G$.

It is well-known that a directed Hamiltonian path exists for any
acyclic orientation of the complete graph $K_n$ on $n$ vertices.
This implies that $d_{\min}(K_n)=d_{\max}(K_n)=\frac 12
(n-1)(n-2)$, hence $K_n$ is fully orientable (\cite{west}).
Throughout this paper, we use $C_n=v_0v_1 \cdots v_{n-1}v_0$ to
represent a cycle of length $n \geqslant 3$. It is  easy to see
that $C_n^2\cong K_n$ if $3\leqslant n\leqslant 5$, and hence is
fully orientable. If $n=6$, then $C_n^2\cong K_{3(2)}$. By the
result of \cite{chan}, $C_6^2$ is not fully orientable and
$d(D)\in \{4,6,7\}$ for any orientation $D$ of $C_6^2$. In this
paper, we shall prove that $C_n^2$ is fully orientable except
$n=6$.

%
\section{Results}
%

For a given graph $G$, let $\pi_{{}_T}(G)$ be the minimum number
of edges that can be deleted from $G$ so that the new graph is
triangle-free, i.e., having no $K_3$ as a subgraph.  The following
lemma appeared in \cite{lai1}.

\begin{lemma}\label{lem1}
For any graph $G$, $d_{\min}(G)\geqslant \pi_{{}_T}(G)$.
\end{lemma}

\begin{lemma}\label{lem2}
For $n\geqslant 7$, $\pi_{{}_T}(C_n^2)=\lceil \frac n2\rceil$.
\end{lemma}

\begin{proof}\ When $n\geqslant 7$, $C_n^2$ contains exactly $n$ distinct
triangles. Since every edge of $C_n^2$ lies in at most two triangles,
we have $\pi_{{}_T}(C_n^2)\geqslant \lceil \frac n2\rceil$.

On the other hand, let $S=\{v_1v_2,v_3v_4,\ldots ,v_{n-1}v_0\}$ if
$n$ is even, and $S=\{v_0v_1,v_1v_2$, $v_3v_4,\ldots$ ,
$v_{n-2}v_{n-1}\}$ if $n$ is odd. Obviously, $|S|=\lceil \frac
n2\rceil$ and $G-S$ is triangle-free. Thus,
$\pi_{{}_T}(C_n^2)\leqslant |S|=\lceil \frac n2\rceil$. 
\end{proof}

\bigskip

In a digraph $D$ with vertex set $V(D)$ and arc set $E(D)$, we
use $u\to v$ to denote the arc with tail $u$ and head $v$. The {\em
indegree} $d_D^-(v)$ of a vertex $v$ in $D$ is the number of arcs
with head $v$; the {\em outdegree} $d_D^+(v)$ of  $v$ in $D$ is
the number of arcs with tail $v$. Let $R(D)$ denote the set of
dependent arcs in $D$.

\begin{theorem}\label{thm1}
If $n\geqslant 7$, then $d_{\min}(C_n^2)=\pi_{{}_T}(C_n^2)+1$.
\end{theorem}

\begin{proof}  In the first part, we are going to prove that
$d_{\min}(C_n^2) \geqslant \pi_{{}_T}(C_n^2)+1$. Assume to the
contrary that $d_{\min}(C_n^2)< \pi_{{}_T}(C_n^2)+1$. It follows
from Lemmas \ref{lem1} and \ref{lem2} that
$d_{\min}(C_n^2)=\pi_{{}_T}(C_n^2)= \lceil \frac n2\rceil$. Let $D$
be an acyclic orientation of $C_n^2$ with $d(D)=d_{\min}(C_n^2)$.
Let $F$ be the set of all underlying edges of the arcs in $R(D)$.
Thus, $|F|=\pi_{{}_T}(C_n^2)=\lceil \frac n2 \rceil$ and $C_n^2-F$
is triangle-free. We use $C$ to denote the closed walk
$v_0,v_1,\ldots ,v_{n-1},v_0$ in $D$.

The proof is divided into two cases, depending on the parity of $n$.

\medskip
\noindent{\bf Case 1.}\   Assume $n=2k$ for some $k\geqslant 4$.
\medskip

Since $C_n^2-F$ is triangle-free and $|F|=k$, it is easy to see from
the construction of $C_n^2$ that $F=\{v_1v_2,v_3v_4,\ldots ,
v_{n-1}v_0\}$ or $F=\{v_0v_1,v_2v_3,\ldots ,$ $v_{n-2}v_{n-1}\}$.
Without loss of generality, we may assume the former.

\medskip
\noindent{\bf Claim.}\  No $v\in V(D)$ satisfies
$d^+_C(v)=d^-_C(v)=1$.
\medskip

Assume to the contrary that we had $v_{i-1}\to v_i \to v_{i+1}$
(indices modulo $n$) for some $i$ in $D$. Then we would have $v_{i-1}
\to v_{i+1}$ in $D$ since $D$ is acyclic, and hence $v_{i-1}\to
v_{i+1}$ is dependent, contradicting the assumption that $v_{i-1}v_{i+1}
\notin F$.

\medskip

It follows from  the Claim  that every vertex $v\in V(D)$
satisfies $d^+_C(v)=0$ and $d^-_C(v)=2$ or $d^+_C(v)=2$ and
$d^-_C(v)=0$. Without loss of generality, we may suppose that
$d^+_C(v_i)=2$ and $d^-_C(v_i)=0$ for each odd $i$, and
$d^+_C(v_i)=0$ and $d^-_C(v_i)=2$ for each even $i$, i.e., $C$ is
oriented as $v_1 \to v_2\leftarrow v_3\to v_4\leftarrow  \cdots
\to v_{n-2}\leftarrow v_{n-1}\to v_0\leftarrow v_1$. Therefore,
$R(D)=\{v_1 \to v_2, v_3 \to v_4,\ldots , v_{n-1}\to v_0\}$. Since
$v_{i}\leftarrow v_{i+1}$ is not dependent for each even $i$, the
edge $v_{i-1}v_{i+1}$ must be directed as $v_{i-1}\to v_{i+1}$.
Consequently, a directed cycle $v_1 \to v_3 \to \cdots \to v_{n-1}
\to v_1$ is constructed, contradicting   the acyclicity of $D$.

\medskip
\noindent{\bf Case 2.}\  Assume $n=2k+1$ for some $k\geqslant 3$.
\medskip

In this case, $|R(D)|=|F|=k+1$. Note that $C_n^2-F$ is
triangle-free. Hence, $C_{n}^2$ has exactly $2k+1$ distinct triangles,
and every edge $v_iv_j$ belongs to exactly one triangle (or two
triangles) depending on the distance between $v_i$ and $v_j$ is
$2$ (or $1$) in $C$, then $F$ contains at least $k$ edges in $C$
and hence at most one edge outside $C$. Since $n$ is odd, there
must exist some $i$ such that $v_{i-1}\to v_i \to v_{i+1}$. So,
$v_{i-1} \to v_{i+1}$ is a dependent arc. Hence,   $F$ contains
exactly one edge outside $C$. We may assume that
$F=\{v_1v_2,v_3v_4$, $\ldots$, $v_{n-2}v_{n-1},$ $v_{n-1}v_1\}$.
We may also assume that $v_{n-1}\to v_1$. (The case for $v_1 \to
v_{n-1}$ can be handled in a similar way.)

We examine the direction of $v_1v_2$. First assume that $v_1\to
v_2$. Since $v_1\to v_2$ is the only arc of $F$ in the triangle $v_1
v_2 v_3 v_1$, we have $v_1 \to v_3 \to v_2$. Similarly, in the
triangle $v_2 v_3 v_4 v_2$, we have $v_2 \to v_4$ and $v_3\to v_4$.
In this way, it leads to $v_4\leftarrow v_5\to v_6 \leftarrow$
$\cdots$  $\leftarrow v_{n-2}\to v_{n-1} \leftarrow v_0$. If $v_1\to
v_0$, then a directed 3-cycle $v_1 \to v_0 \to v_{n-1} \to v_1$ is
produced. If $v_0\to v_1$, then $v_0\to v_1$ would be a dependent
arc of $D$. Contradictions are obtained in both cases.

Next, assume that $v_2\to v_1$. Since $v_2\to v_1$ is the only
arc of $F$ in the triangle $v_1 v_2 v_3 v_1$, we have $v_2 \to v_3 \to
v_1$. Similarly, in the triangle $v_2 v_3 v_4 v_2$, we have $v_4 \to
v_2$ and $v_4\to v_3$. In this way, it leads to $v_3\leftarrow v_4\to
v_5 \to v_6\leftarrow $ $\cdots$ $\leftarrow v_{n-1}\to v_{0}$.
Since $v_{n-1} \to v_0$ is not dependent, we have $v_0\to v_1$.
Since $v_2 \to v_3$ is not dependent, we have $v_3\to v_1$.
Similarly, we have $v_5\to v_3$; $v_7\to v_5$; $\cdots$ ; $v_0\to
v_{n-2}$; $v_2\to v_0$; $v_4\to v_2$;  $v_6\to v_4$; $\cdots$ ;
$v_{n-1}\to v_{n-3}$. However, the existence of the directed path
$v_0 \to v_{n-2} \to v_{n-4} \cdots \to v_5\to v_3 \to v_1$ makes
$v_0 \to v_1$ a dependent arc, contrary to our assumption.

\medskip

In the second part, we are going to prove that $d_{\min}(C_n^2)\leqslant
\pi_{{}_T}(C_n^2)+1$. In fact, an acyclic orientation $D_0$ of $G$ will
be constructed so that $d(D_0)=\pi_{{}_T}(C_n^2)+1$. The construction is
divided into two cases, depending on the parity of $n$.

\medskip
\noindent{\bf Case 1.}\  Assume $n=2k$ for some $k\geqslant 4$.
\medskip

Let $D_0$ be defined as follows.

$v_1 \to v_{n-1}$;  \  $v_1\to v_0$;  \ $v_2\to v_0 \to v_{n-1}$;
\ $v_{n-2} \to v_0$;

$v_{2i-1}\to v_{2i+1}$ for each $i=1,2,\ldots ,k-1$;

$v_{2i}\to v_{2i+2}$ for each $i=1,2,\ldots ,k-2$;

$v_{2i-1}\leftarrow v_{2i}\to v_{2i+1}$ for each $i=1,2,\ldots ,k-1$.

By a close examination, we can see that $D_0$ is an acyclic orientation
of $C_n^2$ such that $R(D_0)=\{v_1 \to v_{n-1}, v_2 \to v_0, v_2 \to v_3,
v_4 \to v_5, \ldots , v_{n-2} \to v_{n-1}\}$. Therefore, $d(D_0)=|R(D_0)|=k+1=
\pi_{{}_T}(C_n^2)+1$.

\medskip
\noindent{\bf Case 2.}\  Assume $n=2k+1$ for some $k\geqslant 3$.
\medskip

Let $D_0$ be defined as follows.

$v_2\to v_1\to v_{n-1}$; \ $v_3\to v_1  \to v_0$; \ $v_2 \to v_0 \to
v_{n-1}$; \ $v_{n-2} \to v_{n-1} $; \ $ v_{n-2} \to v_{0}$;

$v_{2i+1}\to v_{2i}\to v_{2i+2}$ for each $i=1,2,\ldots ,k-1$;

$v_{2i}\leftarrow v_{2i-1}\to v_{2i+1}$ for each $i=2,\ldots ,k-1$.

By a close examination, we can see that $D_0$ is an acyclic orientation
of $C_n^2$ such that $R(D_0)=\{v_3 \to v_1, v_1 \to v_{n-1}, v_2 \to v_0,
v_3 \to v_4, v_5 \to v_6, v_7 \to v_8, \ldots , v_{n-2} \to v_{n-1}\}$.
Therefore, $d(D_0)= |R(D_0)|=k+2=\pi_{{}_T}(C_n^2)+1$. This completes our
proof. 
\end{proof}

\begin{theorem}\label{thm2}
For $n\geqslant 7$, $C_n^2$ is fully orientable.
\end{theorem}

\begin{proof} For every graph $G$, there exists an acyclic orientation
$D$ so that $d(D)=d_{\max}(G)$ in \cite{fish}. So the present
theorem is established if, for each integer $s$,
$\pi_{{}_T}(C_n^2)+1=m < s\leqslant n$, an acyclic orientation
$D_{s-m}$ of $C_n^2$ is constructed to satisfy $d(D_{s-m})=s$. In
fact, such a sequence of acyclic orientations $D_{s-m}$ can be
recursively constructed from the $D_0$ defined in the proof of
Theorem \ref{thm1}. We divide our construction into two cases,
depending on the parity of $n$.

\medskip
\noindent{\bf Case 1.}\  Assume $n=2k$ for some $k\geqslant 4$.
\medskip

By Lemma \ref{lem2}, $\pi_{{}_T}(C_n^2)=k$. First consider the range
$k+2 \leqslant s \leqslant 2k-2$. Assume that $D_{s-k-2}$ has
already been constructed.

Let $D_{s-k-1}$ be the acyclic orientation of $C_n^2$ obtained
from $D_{s-k-2}$ by reversing the arc $v_{2(s-k-1)-1} \to
v_{2(s-k-1)+1}$. It is easy to check that $R(D_{s-k-1})=R(D_{s-k-2})
\cup \{v_{2(s-k-1)} \to v_{2(s-k-1)-1}\}$. Hence $d(D_{s-k-1})$ 
$=d(D_{s-k-2})+1=s$.

If $s=2k-1$, let $D_{k-2}$ be the acyclic orientation of $C_n^2$
obtained from $D_{k-3}$ by reversing the arcs $v_1 \to v_{2k-1}, v_1
\to v_{0}, v_{2k-3} \to v_{2k-1}$. It is easy to check that
$R(D_{k-2})=R(D_{k-3})\setminus \{v_1 \to v_{2k-1}\}\cup \{v_{0} \to
v_{1}, v_{2k-2}\to v_{2k-3}\}$ and $d(D_{k-2})=d(D_{k-3})+2-1
=2k-1$.

If $s=2k$, let $D_{k-1}$ be the acyclic orientation of $C_n^2$
obtained from $D_{k-2}$ by reversing the arc $v_{2k-1} \to v_{1}$.
It is easy to check that $R(D_{k-1})=R(D_{k-2})\setminus \{v_{0}
\to v_1\}\cup \{v_{0} \to v_{2k-1}, v_{2k-5} \to v_{2k-3}\}$ and
$d(D_{k-1})=d(D_{k-2})+2-1=2k$.

\medskip
\noindent{\bf Case 2.}\  Assume $n=2k+1$ for some $k\geqslant 3$.
\medskip

By Lemma \ref{lem2}, $\pi_{{}_T}(C_n^2)=k+1$. First consider the
range $k+2 \leqslant s \leqslant 2k-2$. Assume that $D_{s-k-3}$ has
already been constructed.

Let $D_{s-k-2}$ be the acyclic orientation of $C_n^2$ obtained
from $D_{s-k-3}$ by reversing the arc $v_{2(s-k-2)} \to
v_{2(s-k-2)+2}$. It is easy to check that $R(D_{s-k-2})=R(D_{s-k-3})
\cup \{v_{2(s-k-2)+1} \to v_{2(s-k-2)}\}$. Hence $d(D_{s-k-2})$ 
$=d(D_{s-k-3})+1=s$.

If $s=2k$, let $D_{k-2}$ be the acyclic orientation of $C_n^2$
obtained from $D_{k-3}$ by reversing the arcs $v_{0} \to v_{2k}$. It
is easy to check that $R(D_{k-2})=R(D_{k-3})\setminus \{v_1 \to
v_{2k}\}\cup \{v_{2k-1} \to v_{0},$ $v_{1} \to v_{0}\}$ and
$d(D_{k-2})=d(D_{k-3})-1+2=2k$.

If $s=2k+1$, let $D_{k-1}$ be the acyclic orientation of $C_n^2$
obtained from $D_{k-3}$ by reversing the arc $v_{2k-2} \to
v_{2k}$. It is easy to check that $R(D_{k-1})=R(D_{k-3})\cup
\{v_{2k-1} \to v_{2k-2}$, $v_{2k-4} \to v_{2k-2}\}$ and
$d(D_{k-1})=d(D_{k-3})+2=2k+1$. 
\end{proof}

\bigskip

To conclude this paper, we would like to pose the following problem.

\medskip

\noindent {\bf Problem}.\ For any given integer $k\geqslant 2$,
does there exist a smallest constant $\alpha(k)$ such that $C_n^k$
is fully orientable whenever $n\geqslant \alpha(k)$?

\medskip

Since $C^k_{2k+2}\cong K_{(k+1)(2)}$ and $K_{(k+1)(2)}$ is not
fully orientable when $k\ge 2$ by the result of \cite{chan}, we have
$\alpha(k) \geqslant 2k+3$. The only solved case of this problem
is $\alpha(2)=7$ by Theorem \ref{thm2}.


\end{document}